\documentclass[12pt,a4paper]{article}
\usepackage[utf8]{inputenc}
\usepackage[affil-it]{authblk}
\usepackage{amsmath}
\usepackage{amsthm}
\usepackage{bbm}
\usepackage[top=3cm,left=3cm,right=3cm,bottom=3cm]{geometry}
\usepackage{txfonts}
\usepackage{amsfonts}
\usepackage{amssymb}
\usepackage{enumerate}
\usepackage{enumitem}
\usepackage{verbatim}
\usepackage{graphicx}
\usepackage{wrapfig}
\usepackage{tikz}
\usepackage{tikz-cd}

\allowdisplaybreaks

\newcommand{\NN}{\mathbb N}
\newcommand{\IN}{\mathbb Z}
\newcommand{\RN}{\mathbb R}
\newcommand{\QN}{\mathbb Q}
\newcommand{\CN}{\mathbb C}

\newcommand{\mi}{\mathrm i}

\newcommand{\SP}[2]{\left\langle {#1} , {#2} \right\rangle}

\newcommand{\ket}[1]{\left|{{#1}}\right>}
\newcommand{\norm}[1]{\left|\left|{#1}\right|\right|}
\newcommand{\fock}[1]{\mathcal{F}_{{#1}}}

\newcommand{\Spec}[1]{\mathrm{\sigma}\left({#1}\right)}
\newcommand{\Sph}[1]{\mathbb{S}^{{#1}}}
\newcommand{\tang}[2]{\mathrm{T}_{{#1}}{{#2}}}

\newcommand{\intd}{\mathrm d}
\newcommand{\abs}[1]{\left|{#1}\right|}
\newcommand{\inv}[1]{{#1}^{-1}}
\newcommand{\convol}[2]{\left({#1}*{#2}\right)}
\newcommand{\vspanC}[2]{\text{span}_\CN\left\{{#1}~\middle|~{#2}\right\}}
\newcommand{\fourier}{\mathcal F}

\newcommand{\symp}{\poisson}
\newcommand{\poisson}[2]{\left[{#1},{#2}\right]}

\newcommand{\sympf}[2]{\omega\left({#1},{#2}\right)}

\newcommand{\LP}{\Delta_\bp}
\newcommand{\LQ}{\Delta_\bq}

\newcommand{\Lp}[2]{\mathcal{L}^{{#1}}_{{#2}}}
\newcommand{\Lpn}[3]{\norm{{#1}}_{\Lp{{#2}}{{#3}}}}

\newcommand{\ddim}{d}
\newcommand{\VF}[1]{\hat{{#1}}}

\newcommand{\HamVec}[1]{X_{{#1}}}
\newcommand{\HF}{\mathcal{H}}

\newcommand{\HVF}{\mathcal{H}_{\text{Vl.}}}
\newcommand{\HHF}{\mathcal{H}_{\text{Ht.}}}

\newcommand{\PHHF}{\bar{\mathcal{H}}_\text{Ht.}}
\newcommand{\RX}{{\RN^{\ddim}_{\bx}}}

\newcommand{\RV}{{\RN^{\ddim}_{\bv}}}
\newcommand{\RW}{{\RN^{2\ddim}_{\bw}}}
\newcommand{\RXi}{{\RN^{\ddim}_{\xi}}}
\newcommand{\RZ}{{\RN^{2\ddim}_{\bz}}}

\newcommand{\RZH}{{\RN^{2\ddim}_{\VF{\bz}}}}

\newcommand{\RNp}{\RN_{\geq 0}}
\newcommand{\naX}{\nabla_{\bx}}
\newcommand{\naV}{\nabla_{\bv}}

\newcommand{\naXi}{\nabla_{\xi}}

\newcommand{\del}[2]{\frac{\partial {#1}}{\partial {#2}}}
\newcommand{\Del}[1]{\mathrm D^{#1}}

\newcommand{\bx}{\textbf x}
\newcommand{\bv}{\textbf v}
\newcommand{\bw}{\textbf w}
\newcommand{\bz}{\textbf z}
\newcommand{\by}{\textbf y}

\newcommand{\be}{\textbf e}
\newcommand{\bp}{\textbf p}
\newcommand{\bq}{\textbf q}
\newcommand{\bn}{\textbf 0}
\newcommand{\ba}{\textbf a}
\newcommand{\bb}{\textbf b}
\newcommand{\bc}{\textbf c}

\newcommand{\twovec}[2]{\left(\begin{array}{c} {#1} \\ {#2} \end{array}\right)}

\newcommand{\qma}{\mathfrak{a}}
\newcommand{\qmb}{\mathfrak{b}}
\newcommand{\qmN}{\mathfrak{N}}

\newcommand{\XV}{\mathfrak{X}}
\newcommand{\YV}{\mathfrak{Z}}
\newcommand{\domN}{\mathcal{V}}
\newcommand{\perB}[1]{\mathcal{W}_{{#1}}}

\newcommand{\InDec}[1] {\mathrm{ind}_{\text{dec.}}\left({{#1}}\right)}
\newcommand{\InOsc}[1] {\mathrm{ind}_{\text{osc.}}\left({{#1}}\right)}

\newcommand\pgfmathsinandcos[3]{%
  \pgfmathsetmacro#1{sin(#3)}%
  \pgfmathsetmacro#2{cos(#3)}%
}
\newcommand\LongitudePlane[3][current plane]{%
  \pgfmathsinandcos\sinEl\cosEl{#2} 
  \pgfmathsinandcos\sint\cost{#3} 
  \tikzset{#1/.style={cm={\cost,\sint*\sinEl,0,\cosEl,(0,0)}}}
}
\newcommand\LatitudePlane[3][current plane]{%
  \pgfmathsinandcos\sinEl\cosEl{#2} 
  \pgfmathsinandcos\sint\cost{#3} 
  \pgfmathsetmacro\yshift{\cosEl*\sint}
  \tikzset{#1/.style={cm={\cost,0,0,\cost*\sinEl,(0,\yshift)}}} %
}

\newcommand\DrawLongitudeHalfCircle[2][1]{
  \LongitudePlane{\angEl}{#2}
  \tikzset{current plane/.prefix style={scale=#1}}
  \pgfmathsetmacro\angVis{atan(sin(#2)*cos(\angEl)/sin(\angEl))} %
  \draw[current plane] (\angVis:1) arc (\angVis:90:1);
  \draw[current plane,dashed] (-90:1) arc (-90:\angVis:1);
}
\newcommand\DrawLatitudeCircle[2][1]{
  \LatitudePlane{\angEl}{#2}
  \tikzset{current plane/.prefix style={scale=#1}}
  \pgfmathsetmacro\sinVis{sin(#2)/cos(#2)*sin(\angEl)/cos(\angEl)}
  \pgfmathsetmacro\angVis{asin(min(1,max(\sinVis,-1)))}
  \draw[current plane] (\angVis:1) arc (\angVis:-\angVis-180:1);
  \draw[current plane,dashed] (180-\angVis:1) arc (180-\angVis:\angVis:1);
}

\tikzset{%
  >=latex, 
  inner sep=0pt,%
  outer sep=2pt,%
  mark coordinate/.style={inner sep=0pt,outer sep=0pt,minimum size=3pt,
    fill=black,circle}%
}

\makeindex

\begin{document}


\theoremstyle{plain}
\newtheorem{thm}{Theorem}[section]
\newtheorem{lem}[thm]{Lemma}
\newtheorem{cor}[thm]{Corollary}
\newtheorem{prop}[thm]{Proposition}
\newtheorem{conj}[thm]{Conjecture}

\theoremstyle{definition}
\newtheorem{defn}[thm]{Definition}

\theoremstyle{remark}
\newtheorem{rmk}[thm]{Remark}
\newtheorem{exam}[thm]{Example}

\title{Symmetry reduction and periodic solutions in Hamiltonian Vlasov systems}
\author{R.A. Neiss
	\thanks{
		Electronic address: \texttt{rneiss@math.uni-koeln.de}
	}
}
\affil{
	Universität zu Köln, Mathematisches Institut, \\
	Weyertal 86-90, 50931 Köln, Germany
}
\date{\today}

\maketitle

\begin{abstract}
\noindent In this paper, we discuss a general approach to find periodic solutions bifurcating from equilibrium points of classical Vlasov systems. The main access to the problem is chosen through the Hamiltonian representation of any Vlasov system, firstly put forward in \cite{froehlichknowlesschwarz} and generalized in \cite{neiss,neisspickl}. The method transforms the problem into a setup of complex valued $\mathcal{L}^2$ functions with phase equivariant Hamiltonian. Through Marsden-Weinstein symmetry reduction \cite{marsdenweinsteinreduction}, the problem is mapped on a Hamiltonian system on the quotient manifold $\mathbb{S}^{\mathcal{L}^2}/\mathbb{S}^1$ which actually proves to be necessary to close many trajectories of the dynamics. As a toy model to apply the method we use the Harmonic Vlasov system, a non-relativistic Vlasov equation with attractive harmonic two-body interaction potential. The simple structure of this model allows to compute all of its solutions directly and therefore test the benefits of the Hamiltonian formalism and symmetry reduction in Vlasov systems.
\end{abstract}

\section{Introduction}

The problem of finding periodic solutions to Vlasov type equations has been around for a couple of decades. While there exist some results for solutions on periodic domains \cite{battrein}, or periodic solutions under boundary conditions \cite{bostan}, the existence of periodic solutions on the full domain $\RX\times\RV$, in particular $\ddim=3$, for any type of interaction potential seems to be untreated so far. 

The construction of periodic solutions can usually be achieved by two major methods, both of which require a Hamiltonian formulation of the dynamical system.  One option is to find extremal points of some action functional evaluated on a Banach space of closed curves. This is particularly useful, if one has a semi-bounded Hamiltonian, giving a chance to apply some mountain pass techniques.

Alternatively, one can first try to identify stationary points of the dynamics and find non-resonant eigenvalues of the second derivative of the Hamiltonian. This allows to construct families of closed curves oscillating around the stationary point. This bifurcation method is well developed for many finite and infinite dimensional systems \cite{ambrosettiprodi}. Finding and classifying equilibrium states is also an active topic of research, considering for example stability questions \cite{mouhot,rein}. \\

\noindent This paper intends to study the Hamiltonian formulation of Vlasov systems given by the Hamiltonian Vlasov equation \cite{froehlichknowlesschwarz,neiss,neisspickl} with respect to the search of periodic solutions thereof. In the Hamiltonian Vlasov picture, the phase space variable is a complex valued $\Lp{2}{}$ function $\alpha$ on $\RX\times\RV$, s.t. $f\equiv\abs{\alpha}^2\in\Lp{1}{}$ is a classical Vlasov density again. The Hamiltonian functional can be directly constructed from the energy functional of the classical Vlasov system. 

In this formulation a couple of remarkable features are revealed that underline its potential relevance for further studies. For example, it is possible for the first time to identify conserved quantities of the Vlasov system such as mass, and linear and angular momentum as Noether conjugates of continuous symmetries. Encouraged by geometric ideas usually applied to finite-dimensional problems, we study the ideas of symplectic symmetry reduction usually attributed to Marsden-Weinstein \cite{marsdenweinsteinreduction} in this new PDE setup. \\

\noindent The main motivation to remove continuous symmetries is their violation of bifurcation conditions. Symmetries always contribute to the kernel of the second derivative of the Hamiltonian, because they generate continuous families of stationary points and periodic families, degenerating the bifuraction equation. In this paper we discuss different approaches to remove these degeneracies, partly succeeding.

It turns out, that there are two types of symmetries. The first type (phase invariance) acts smoothly on the model Banach space ($\subseteq\Lp{2}{}$) and can be treated by symplectic symmetry reduction following Marsden and Weinstein \cite{marsdenweinsteinreduction}. 
The second type (translation/ rotation invariance) is much more difficult to overcome. The key issue is that the group action is not smooth anymore, only continuous. That defies any chance to apply global symmetry reduction and projection onto a symplectic quotient manifold based on existing results, as a key condition is violated. Nevertheless there is hope to remove them with some local reduction principle around equilibria points if one is able to choose an appropriate topological setup. Still, this remains open as technical difficulties with regularities arise. \\

\noindent Treating the reduction of the phase symmetry, we find that it is not only convenient but even necessary for the search of periodic solutions. 

The phase invariance of the Hamiltonian yields mass conservation and motivates restriction to the sphere $\Sph{\Lp{2}{}}$. Projection onto the quotient space $\Sph{\Lp{2}{}}/\Sph{1}$ not only increases the number of stationary points, but also predicts the correct periodic families bifurcating around them in the studied example of Harmonic Vlasov. Indeed, these bifurcating families are invisible for other methods, as in the unreduced system they are also governed by global phase oscillation usually in irrational relation to the profile oscillation, hence only relatively periodic in the notion of \cite{marsdenweinsteinreduction}. For our simplified example of Harmonic Vlasov (Thm. \ref{thm:harmonic:spectral-parameterization}), we find a one-to-one correspondence of the Hamiltonian's spectrum and families of bifuracting periodic solutions in the direction of eigenvectors. This indicates the method's possible impact on more complicated problems. \\

\noindent The paper is structured as follows. Section \ref{sec:harmonic:symplectic-reduction} gives a short introduction to the principle of symplectic symmetry reduction alongside its application to Hamiltonian PDEs in general. It also discusses the setup of the Hamiltonian Vlasov formulation. In Section \ref{sec:harmonic:harmonic-vlasov} the method is applied to the Harmonic Vlasov system. The paper closes with some final comments on further applications.  \\

\section{Symplectic symmetry reduction}

\label{sec:harmonic:symplectic-reduction}

The idea of symplectic symmetry reduction in the sense of Marsden and Weinstein \cite{marsdenweinsteinreduction} is the reduction of the phase space of a Hamiltonian system according to its continuous symmetries and the associated conserved quantities. This procedure is particularly useful as it removes degeneracies around bifurcation points. For completeness, we give a short overview of the method, before introducing the specific setup of Hamiltonian Vlasov. We remark that the method applies to all Hamiltonian PDEs if they fit into an appropriate topological framework, especially smooth group action is required. Among others, this is the case for phase multiplication in NLS, Hartree, and the Massive Thirring Model.

\subsection{On Hamiltonian PDEs}

Consider a symplectic phase space $(M,\omega)$, for the sake of simplicity assumed to be a linear space $M$ with constant symplectic form $\omega$, with a smooth Hamiltonian $H: M\to\RN$ and the associated Hamiltonian vector field $\HamVec{H}: M\to \tang{}{M}$ implicitly given by
\begin{equation}
\label{eqn:harmonic:ham-vec}
\Del{}H(\alpha)(\delta\alpha) = -\sympf{\HamVec{H}(\alpha)}{\delta\alpha} \quad \forall(\alpha,\delta\alpha)\in \tang{}{M},
\end{equation}
$\tang{}{M}$ denoting the tangent bundle. The Hamiltonian equation of motion is now given by
\begin{equation}
\label{eqn:harmonic:ham-eqn}
\partial_t\alpha(t)  = \HamVec{H}(\alpha(t)).
\end{equation}
Our main interest is to compute equilibria and families of periodic solutions around them. \\

\noindent Let us assume that there is a Lie group $G$, for simplicity assumed to be Abelian and finite-dimensional, acting smoothly on the phase space $(M,\omega)$, 
\begin{equation*}
G \times M \to M, \quad (g,\alpha) \mapsto g.\alpha,
\end{equation*}
leaving the symplectic form invariant, i.e.,
\begin{equation*}
\forall g\in G,~ \alpha,\beta\in M: \quad \sympf{g.\alpha}{g.\beta} = \sympf{\alpha}{\beta}.
\end{equation*}
This is sometimes denoted $G\subseteq\mathrm{Symp}(M)$. Also assume that $G$ is a continuous symmetry of the Hamiltonian, i.e.,
\begin{equation}
\label{eqn:harmonic:ham-symmetries}
\forall g\in G,~ \alpha\in M: \quad H(g.\alpha) = H(\alpha).
\end{equation}

\begin{exam}[Nonlinear Schrödinger equation]
The NLS equation fits into the given framework. The manifold $M=\Lp{2}{x}$ with symplectic form
\begin{equation*}
\sympf{u}{v} \equiv \int_{\RN} u~ \bar{v}~ \intd x,
\end{equation*}
group action
\begin{equation*}
\Sph{1} \times \Lp{2}{x} \to \Lp{2}{x}, \quad (\zeta,u) \mapsto \zeta~ u,
\end{equation*}
and Hamiltonian functional
\begin{equation*}
\HF(u) \equiv \frac{1}{2} \int_{\RN} \left(\abs{u_x}^2 + \frac{1}{2} \abs{u}^4\right)~ \intd x.
\end{equation*}
The Hamiltonian formalism yields the well-known NLS equation
\begin{equation*}
\mi~ u_t = -u_{xx} + \abs{u}^2 u.
\end{equation*}
The symplectic form and the Hamiltonian are also invariant under translation in $x$, but this group action is not smooth, only continuous.
\end{exam}

\noindent The existence of such a symmetry group $G$ raises two major concerns for our analysis. Firstly, it embeds any equilibrium into a continuous family, as shifting by group elements yields more equilibria of the same type. This causes technical problems especially for bifurcation theory, as these families contribute to the kernel of $\Del{2}H$, disallowing to solve 
\begin{equation*}
\lambda~ \partial_t\alpha(t) - \HamVec{H}(\alpha(t)) = 0, \quad \alpha: \RN/\IN \to M, \quad \lambda\in\RNp,
\end{equation*}
locally and thereby prove existence of periodic families for certain frequencies $\lambda$. 

On the other hand, the symmetry operations often are of no physical interest, but can easily destroy periodicity. For example, global $\Sph{1}$ phase oscillation often is not relevant as it does not change the profile of the wave package. The same holds for profiles travelling at constant speed, such as solitons in NLS.

The symplectic symmetry reduction solves both these problems, removing the degeracies of equilibria arising from the action of $G$ and dividing out the non-relevant motion along trajectories of the group action. The newly generated equilibria in the quotient system are then referred to as \textit{relative equilibria}. They still carry all qualitatively interesting information. \\

\noindent The process of symmetry reduction in practice boils down to the following steps.
\begin{enumerate}[label=(\roman*)]
\item From equations \eqref{eqn:harmonic:ham-vec}, \eqref{eqn:harmonic:ham-eqn}, and \eqref{eqn:harmonic:ham-symmetries} one computes the full set of conserved quantities implied by this group (Noether's Theorem). The number is equal to the (real) dimension $q$ of the group $G$. The resulting map
\begin{equation*}
\mathcal{C}: M \to \RN^q.
\end{equation*}
is usually referred to as \textit{moment map}.
\item From the conservation laws, we know that trajectories of \eqref{eqn:harmonic:ham-eqn} remain in level sets of $\mathcal{C}$. Therefore it suffices to restrict ourselves to the level set $\inv{\mathcal{C}}(\bc)$ of a regular value $\bc\in\RN^q$, enforcing it to carry the structure of a manifold.
\item The Marsden-Weinstein reduction now states that there is a smooth quotient map
\begin{equation*}
\Pi: \inv{\mathcal{C}}(\bc) \to \inv{\mathcal{C}}(\bc)/G
\end{equation*}
contracting orbits of the group action, such that $\inv{\mathcal{C}}(\bc)/G$ has a symplectic form $\Omega$, that satisfies for the embedding $\iota: \inv{\mathcal{C}}(\bc) \to M$
\begin{equation*}
\iota^*\omega = \Pi^*\Omega.
\end{equation*}
\item Finally, if $\bar{H}: \inv{\mathcal{C}}(\bc)/G \to \RN$ is chosen to satisfy $H\circ\iota=\bar{H}\circ\Pi$,
\begin{equation*}
\begin{tikzcd}[ampersand replacement=\&]
\inv{\mathcal{C}}(\bc)
\arrow[r, rightarrow, shift left, "\Pi"]
\arrow[d, rightarrow, shift left, "\iota"]
\& 
\left(\inv{\mathcal{C}}(\bc)/G, \Omega\right)
\arrow[d, rightarrow, shift left, "\bar{H}"]
\\
\left(M,\omega\right)
\arrow[r, rightarrow, shift left, "H"]
\&
\RN,
\end{tikzcd}
\end{equation*}
and $\HamVec{\bar{H}}$ is its Hamiltonian vector field w.r.t. $\Omega$, then
\begin{equation*}
\forall\alpha\in \inv{\mathcal{C}}(\bc): \quad \intd\Pi_\alpha(\HamVec{H}(\alpha)) = \HamVec{\bar{H}}(\Pi\alpha).
\end{equation*}
This key relation indicates how to compute the new Hamiltonian vector field locally in charts of the quotient manifold which is otherwise quite challenging in applications. It also shows that motion purely governed by symmetry transformations of $G$ is canceled, yielding new \textit{relative equilibria}, not even detectable in the unreduced system. In charts around these points it is now possible to use bifurcation theory in search of \textit{relatively periodic families}.
\end{enumerate}

\subsection{Vlasov as a Hamiltonian PDE}

While the structural approach of the previous section is very general, we want to study the special case of the Vlasov system in its Hamiltonian form. While the Hamiltonian formalism is widely used to understand equations such as NLS or Hartree, this seems not yet the case for Vlasov. However, the concepts of this section are just as applicable to those other equations. \\

\noindent Let us consider a classical Vlasov system with an energy functional $\HF(f)\in\RN$ for suitable functions $f\in\Lp{1}{\bz}$, $\bz=(\bx,\bv)\in\RZ=\RX\times\RV$. In the special case of a kinetic energy $\epsilon: \RV\to\RN$ and a two-body interaction with potential $\Gamma: \RX\to\RN$, one might think of
\begin{equation}
\label{eqn:harmonic:h-kinetic-twobody}
\HF(f) \equiv \int_{\RZ} \epsilon\left(\bv\right)~ f(\bx,\bv)~ \intd(\bx,\bv) + \frac{1}{2} \int_{\RZ\times\RZ} f(\bz_1)~ \Gamma(\bx_1-\bx_2)~ f(\bz_2)~ \intd(\bz_1,\bz_2).
\end{equation}
As derived in \cite{neiss,neisspickl}, one can find a corresponding Hamiltonian equation on complex-valued functions $\alpha\in\Lp{2}{\bz}$, where the inner product is
\begin{equation*}
\Re\SP{\alpha_1}{\alpha_2} = \Re\int_{\RZ} \alpha_1(\bz)~ \bar{\alpha}_2(\bz)~ \intd\bz,
\end{equation*}
$\Re$ denoting the real part, and with symplectic structure $\mi\in\CN$, the symplectic form is
\begin{equation*}
\sympf{\alpha_1}{\alpha_2} \equiv \Re\SP{\alpha_1}{\mi\alpha_2} = \Im\SP{\alpha_1}{\alpha_2}.
\end{equation*}
The dynamics on $\Lp{2}{\bz}$ as a symplectic vector space can now be constructed with the new Vlasov Hamiltonian
\begin{equation}
\label{eqn:harmonic:vlasov-hamiltonian}
\HVF(\alpha) \equiv \frac{1}{2\mi} \Del{}\HF\left(\abs{\alpha}^2\right) \left(\poisson{\bar{\alpha}}{\alpha}\right),
\end{equation}
$\poisson{\alpha}{\beta}\equiv \naX\alpha\cdot\naV\beta-\naV\alpha\cdot\naX\beta$ denoting the standard Poisson bracket on $\RX\times\RV$. 
The Hamiltonian vector field $\HamVec{\HVF}$ associated to this Hamiltonian, is (if it exists) implicitly defined by
\begin{equation*}
\Del{}\HVF(\alpha)(\delta\alpha) = -\sympf{\HamVec{\HVF}(\alpha)}{\delta\alpha}.
\end{equation*}
The corresponding equation of motion then is
\begin{equation}
\label{eqn:harmonic:hvl} \tag{HVl}
\partial_t\alpha(t) =~ \HamVec{\HVF}(\alpha(t)),
\end{equation}
and in the special case of \eqref{eqn:harmonic:h-kinetic-twobody}, we find
\begin{align*}
\partial_t\alpha(t) =&~ \symp{\epsilon(\bv) + \convol{\Gamma}{\abs{\alpha(t)}^2}(\bx)}{\alpha(t)} - \convol{\Gamma}{\symp{\bar{\alpha}(t)}{\alpha(t)}}~ \alpha(t) \\
=&~ -(\naV\epsilon)(\bv)\cdot \naX\alpha(t) + \convol{\nabla\Gamma}{\abs{\alpha(t)}^2}(\bx)\cdot \naV\alpha(t) - \convol{\Gamma}{\poisson{\bar{\alpha}(t)}{\alpha(t)}}(\bx)~ \alpha(t)
\end{align*}
It is a general fact, that solutions $t\mapsto \alpha(t)$ then yield solutions $t\mapsto f(t)\equiv \abs{\alpha(t)}^2$ of the underlying Vlasov system \cite[Prop.1.1]{neiss}. \\

\noindent If the system energy is structured as in \eqref{eqn:harmonic:h-kinetic-twobody}, then there is a number of general continuous symmetries. These symmetries naturally imply conserved quantities as an example of Noether's Theorem.

\begin{prop}[Continuous symmetries and Noether's Theorem]
\label{prop:harmonic:noether}
Any Vlasov Hamiltonian $\HVF$ structured like \eqref{eqn:harmonic:h-kinetic-twobody} is invariant under the continuous action of the $2\ddim+1$ dimensional Lie group $G=\Sph{1}\times \RX\times\RXi$ given by
\begin{align*}
G \times \Lp{2}{\bz} \to \Lp{2}{\bz}, \quad (g,\alpha) = \left(\left(\zeta,\bx_0,\xi_0\right),\alpha\right) \mapsto \left(g.\alpha\right)(\bx,\bv) \equiv \zeta~\alpha(\bx - \bx_0, \bv)~ \exp\left(\mi\bv\cdot \xi_0\right).
\end{align*}
This implies that any solution $t\mapsto \alpha(t)$ of \eqref{eqn:harmonic:hvl} conserves the $2\ddim+1$ quantities
\begin{align*}
t \mapsto \Lpn{\alpha(t)}{2}{\bz}^2, \quad t \mapsto \SP{\alpha(t)}{\mi~ \naX\alpha(t)}, \quad t\mapsto \SP{\alpha(t)}{\bv~ \alpha(t)},
\end{align*}
if they are defined.
\end{prop}

\begin{proof}
\textbf{(i) Constant phase invariance and mass conservation.} The phase invariance of $\HVF$ is immediate from \eqref{eqn:harmonic:vlasov-hamiltonian}. Let $t\mapsto \alpha(t)$ solve \eqref{eqn:harmonic:hvl}, then
\begin{align*}
\frac{1}{2}~ \partial_t \Lpn{\alpha(t)}{2}{\bz}^2 =&~ \Re\SP{\dot{\alpha}(t)}{\alpha(t)} = \Re\SP{\HamVec{\HVF}(\alpha(t))}{\alpha(t)} = \sympf{\HamVec{\HVF}(\alpha(t))}{\mi\alpha(t)} \\
\stackrel{\text{Def.}}{=}&~ -\Del{}\HVF(\alpha(t))(\mi\alpha(t)) = -\left.\frac{\intd}{\intd\tau}\right|_{\tau=0} \HVF\left(\left(e^{\mi\tau},\bn,\bn\right).\alpha(t)\right) = 0.
\end{align*} \\

\noindent \textbf{(ii) Translation invariance and pseudo momentum conservation.} The translation invariance in $\bx$ is also immediate from \eqref{eqn:harmonic:h-kinetic-twobody}. For any $1\leq i\leq \ddim$ one computes
\begin{align*}
\frac{1}{2} \partial_t\SP{\alpha(t)}{\mi~ \partial_{x_i}\alpha(t)} 
=&~ \Re\SP{\dot{\alpha}(t)}{\mi~ \partial_{x_i}\alpha(t)} 
= \sympf{\HamVec{\HVF}(\alpha(t))}{\partial_{x_i}\alpha(t)} \\
\stackrel{\text{Def.}}{=}&~ - \Del{}\HVF(\alpha(t))(\partial_{x_i}\alpha(t)) = \left.\frac{\intd}{\intd\tau}\right|_{\tau=0} \HVF((1,\tau\be_i,\bn).\alpha(t)) = 0.
\end{align*} \\

\noindent\textbf{(iii) Linear phase invariance and linear momentum conservation.} This symmetry is not directly obvious, though its implication, conservation of linear momentum, is well known. One computes for $\xi_0\in\RXi$
\begin{align*}
\symp{e^{-\mi \bv\cdot\xi_0}\bar{\alpha}}{e^{\mi \bv\cdot\xi_0}\alpha} - \symp{\bar{\alpha}}{\alpha} = \symp{\bar{\alpha}}{\mi \bv\cdot\xi_0} \alpha + \symp{-\mi \bv\cdot\xi_0}{\alpha} \bar{\alpha} = \mi \symp{\abs{\alpha}^2}{ \bv\cdot\xi_0} = \mi~\naX\abs{\alpha}^2 \cdot\xi_0,
\end{align*}
which can be seen to not contribute, neither to the kinetic nor the potential energy term of $\HVF$. The associated conserved quantity is derived as in (i) and (ii).
\end{proof}

\begin{rmk}
\textbf{(i).} For an $\mathrm{SO}(\ddim)$ invariant kinetic energy $\epsilon$ and interaction potential $\Gamma$, also 
\begin{equation*}
t \mapsto \SP{\alpha}{\left(x_i~ \mi~\partial_{x_j} - x_j~ \mi~\partial_{x_i}\right) \alpha}, \quad 1\leq i < j \leq \ddim,
\end{equation*}
are Noether conjugate conserved quantities. \\

\noindent\textbf{(ii).} Opposed to classical mechanics, momentum conservation here is not a consequence of translation invariance, but rather implied by linear phase invariance. This linear phase invariance is actually turns into translation invariance in the Fourier conjugate of the velocity coordinates.
\end{rmk}

\begin{rmk}[Moment map]
\label{rmk:harmonic:moment-map}
In the geometric literature, the key quantity for the symmetry reduction principle is the moment map. It is usually defined as a map $M\to\mathfrak{g}^*$ with values in the dual space of the Lie algebra $\mathfrak{g}=\tang{1_G}{G}$.

We want to show that in our application the moment map is equivalent to the conserved quantities from the previous Proposition. Given $G=\Sph{1}\times\RX\times\RXi$, define for any $(s,\by,\eta)\in\RN\times\RX\times\RXi=\mathfrak{g}$ the vector field
\begin{equation*}
X_{(s,\by,\eta)}(\alpha) \equiv \left.\del{}{\tau}\right|_{\tau=0} \exp\left(\tau(s,\by,\eta)\right).\alpha = \left(\mi s + \by\cdot\naX + \mi \eta\cdot\bv\right)\alpha
\end{equation*}
given through the derivative of the exponential map $\exp: \mathfrak{g} \to G$. In this setup, each of these vector fields is Hamiltonian in the sense, that it is generated by a real valued function on (a dense subspace of) $\Lp{2}{\bz}$, explicitly given by
\begin{equation*}
- \sympf{X_{(s,\by,\eta)}(\alpha)}{\delta\alpha} = \Re\SP{s~ \alpha + \by \cdot \mi\naX\alpha + \eta\cdot\bv~ \alpha}{\delta\alpha} = \Del{}\mathcal{C}_{(s,\by,\eta)}(\alpha)(\delta\alpha)
\end{equation*}
for
\begin{equation*}
\mathcal{C}_{(s,\by,\eta)}(\alpha) = \frac{s}{2} \Lpn{\alpha}{2}{\bz}^2 + \frac{\by}{2} \cdot \SP{\alpha}{\mi\naX\alpha} + \frac{\eta}{2} \cdot \SP{\alpha}{\bv\alpha} = 
\frac{1}{2} \left(\begin{array}{c}
\Lpn{\alpha}{2}{\bz}^2 \\ \SP{\alpha}{\mi~\naX\alpha} \\ \SP{\alpha}{\bv~ \alpha}
\end{array}\right) \cdot \left(\begin{array}{c}
s \\ \by \\ \eta
\end{array}\right),
\end{equation*}
which for every $\alpha$ is linear in $(s,\by,\eta)$ and naturally defines an elements in $\mathfrak{g}^*$. Hence, under a linear isomorphism $\mathfrak{g}^*\simeq\RN^{2\ddim+1}$, a vector of the conserved quantities from Proposition \ref{prop:harmonic:noether} equals the moment map, justifying our choice of notation.
\end{rmk}

\noindent At first sight, the phase invariance is an unwanted degeneracy, because it does not seem to reflect physical properties of the classical Vlasov system. All phase information will be lost upon the mapping $\alpha\mapsto\abs{\alpha}^2$ anyway. 

Nevertheless, it is not artificial at all since, in this Hamiltonian formulation, it provides two continuous symmetries that are the Noether conjugates of mass and momentum conservation. \\

\noindent Following the guidelines of Marsden-Weinstein reduction, it is desirable to cancel out the group action, which in the dynamics corresponds just to phase oscillation ($\ddim+1$) and translation ($\ddim$). The huge technical problem here is that the group action is not smooth. Smoothness is required to secure the quotient space to be a smooth manifold and to define its differentiable structure. It seems that this obstruction is not easily removed. From the structure of the derivative it is also obvious that this is not possible for a space of non-smooth functions.

Here, we restrict ourselves to applying the reduction only with respect to global phase multiplication which is smooth. Clearly, the global phase multiplication is invisible in the classical Vlasov picture after mapping solutions under $\alpha\mapsto\abs{\alpha}^2$.

\begin{prop}[Phase equivariant symplectic reduction]
\label{prop:harmonic:pe-symplecticity}
Consider the symplectic vector space $\left(\Lp{2}{\bz},\omega\right)$ with the group action
\begin{equation*}
\Sph{1} \times \Lp{2}{\bz} \to \Lp{2}{\bz}, \quad (\zeta,\alpha) \mapsto \zeta~ \alpha.
\end{equation*}
Given its associated moment map
\begin{equation*}
\mathcal{C}: \Lp{2}{\bz} \mapsto \RN, \quad \alpha \mapsto \frac{1}{2}\Lpn{\alpha}{2}{\bz}^2,
\end{equation*}
$1/2\in\RN$ is a regular value. By Marsden-Weinstein reduction, there is a unique symplectic form $\bar{\omega}$ on the smooth Hilbert manifold $\Sph{\Lp{2}{\bz}}/\Sph{1}$, s.t. under the embedding and projection maps
\begin{equation*}
\iota: \Sph{\Lp{2}{\bz}} \to \Lp{2}{\bz} \quad \text{and} \quad \Pi: \Sph{\Lp{2}{\bz}} \to \Sph{\Lp{2}{\bz}}/\Sph{1}
\end{equation*}
the relation
\begin{equation*}
\iota^*\omega = \Pi^*\bar{\omega}
\end{equation*}
holds. Also, given a $\Sph{1}$-invariant Hamiltonian $\HF$ on $\Lp{2}{\bz}$ and its push-forward $\bar{\HF}$ on $\Sph{\Lp{2}{\bz}}/\Sph{1}$, the respective Hamiltonian vector fields satisfy
\begin{equation*}
\forall \alpha\in\Sph{\Lp{2}{\bz}}: \quad
\HamVec{\bar{\HF}}(\Pi\alpha) = \left(\Pi_*\HamVec{\HF}\right)(\Pi\alpha) = \intd\Pi_\alpha(\HamVec{\HF}(\alpha)).
\end{equation*}
\end{prop}

\begin{proof}
The group action is smooth, free, and proper (since $\Sph{1}$ is compact). It is also symplectic as it leaves the symplectic form $\omega$ invariant. The moment map is constructed as in Remark \ref{rmk:harmonic:moment-map}. The rest is straightforward application of the reduction theorem \cite[Thm.1]{marsdenweinsteinreduction}. The claim on the Hamiltonian vector fields is just an application of \cite[Cor.3]{marsdenweinsteinreduction}.
\end{proof}

\noindent With explicit computations in mind, it is useful to give coordinate representations for the manifolds and important maps in between.

\begin{lem}[Coordinate representations]
\label{lem:harmonic:coordinate-representations}
We have the following explicit coordinate representations:
\begin{enumerate}[label=(\roman*)]
\item $\tang{}{\Sph{\Lp{2}{\bz}}}$ can be parametrized as a subset of $\Lp{2}{\bz}$ by
\begin{equation*}
\tang{}{\Sph{\Lp{2}{\bz}}} \simeq \left\{(\alpha,\delta\alpha)\in \left(\Lp{2}{\bz}\right)^2 : \Lpn{\alpha}{2}{\bz}=1, \Re\SP{\alpha}{\delta\alpha}=0 \right\},
\end{equation*}
\item its quotient space $\tang{}{\left(\Sph{\Lp{2}{\bz}}/\Sph{1}\right)}$ by dividing out the compact Lie group $\Sph{1}$ allows the natural chart
\begin{equation*}
\tang{}{\left(\Sph{\Lp{2}{\bz}}/\Sph{1}\right)} \simeq \left\{\begin{array}{c}
\Sph{1}(\alpha, \delta\alpha) \in \left(\Sph{\Lp{2}{\bz}} \times \Lp{2}{\bz}\right)/\Sph{1} : \\ \SP{\alpha}{\delta\alpha} = 0 \end{array}\right\},
\end{equation*}
\item yielding for the projection map
\begin{equation*}
\Pi: \begin{array}{ccc}
\tang{}{\Sph{\Lp{2}{\bz}}} & \to & \tang{}{\left(\Sph{\Lp{2}{\bz}}/\Sph{1}\right)} \\
\left\{\begin{array}{c}
(\alpha,\delta\alpha)\in \left(\Lp{2}{\bz}\right)^2 : \\
\Lpn{\alpha}{2}{\bz}=1, \Re\SP{\alpha}{\delta\alpha}=0 \end{array}\right\} & \to & \left\{\begin{array}{c}
\Sph{1}(\alpha, \delta\alpha) \in \left(\Sph{\Lp{2}{\bz}} \times \Lp{2}{\bz}\right)/\Sph{1} : \\ \SP{\alpha}{\delta\alpha} = 0 \end{array}\right\} \\ & & \\
(\alpha,\delta\alpha) & \mapsto & (\Pi\alpha,\intd\Pi_\alpha\delta\alpha) = \Sph{1}\left(\alpha, \delta\alpha - \SP{\delta\alpha}{\alpha}\alpha\right).
\end{array}
\end{equation*}
\item Finally, the symplectic form $\bar{\omega}$ on $\Sph{\Lp{2}{\bz}}/\Sph{1}$ is given by
\begin{equation*}
\bar{\omega}: \begin{array}{ccc}
\left\{\begin{array}{c}
\Sph{1}(\alpha, \delta\alpha_1,\delta\alpha_2) \in \left(\Sph{\Lp{2}{\bz}} \times \left(\Lp{2}{\bz}\right)^2\right)/\Sph{1} : \\ \SP{\alpha}{\delta\alpha_1} = \SP{\alpha}{\delta\alpha_2} = 0 \end{array}\right\} & \to & \RN \\ & & \\
\Sph{1}(\alpha, \delta\alpha_1, \delta\alpha_2) & \mapsto & \Im\SP{\delta\alpha_1}{\delta\alpha_2}.
\end{array}
\end{equation*}
\end{enumerate}
\end{lem}

\begin{proof}
All claims are self-evident.
\end{proof}

\subsection{Remarks on bifurcation}

The main motivation behind the non-trivial transformations made in the previous section is to simplify the search for stationary points and bifurcating periodic solutions around them, as they are known to yield periodic solutions for the underlying classical Vlasov system. \\

\noindent The method explored has two notable advantages. Firstly, switching from the classical $\Lp{1}{}$ formulation to a Hamiltonian $\Lp{2}{}$ language largely enhances access to methods relying on spectral theory. This links the Vlasov systems to a rich field of periodic bifurcations previously applied to many Hamiltonian PDEs \cite{ambrosettiprodi}.

Secondly, cancelling the artificial symmetry of phase equivariance by projection onto the quotient manifold $\Sph{\Lp{2}{}}/\Sph{1}$ actually increases the potential number of stationary points and periodic solutions of Hamiltonian Vlasov, because pure uniform phase oscillation is now invisible to the dynamics. This has the great advantage that for any relative equilibrium on the quotient manifold, one can now try to solve the bifurcation equation
\begin{equation*}
\lambda~ \partial_t\alpha(t) - \HamVec{\bar{\HF}}(\alpha(t)) = 0, \quad \alpha: \RN/\IN \to \Sph{\Lp{2}{}}/\Sph{1}, \quad \lambda>0,
\end{equation*}
in a small neighborhood on the tangent bundle $T\left(\Sph{\Lp{2}{}}/\Sph{1}\right)$, locally linearized to a $\CN$-codimension 1 subspace of $\Lp{2}{}$. Taking care of the phase oscillation degeneracy on the full $\Lp{2}{}$ space is topologically complicated.

\section{Example: Harmonic Vlasov system}
\label{sec:harmonic:harmonic-vlasov}

As a toy model, where the developed toolbox works almost to perfection, we want to classify the periodic solutions of the \textit{Harmonic Vlasov system}. This is the system of non-relativistic motion with attractive harmonic two-body interaction potential
\begin{equation*}
\Gamma(\bx) \equiv \frac{\abs{\bx}^2}{2}.
\end{equation*}
The system's energy functional then is
\begin{equation}
\label{eqn:harmonic:vl-energy}
\HF(f) = \int_{\RZ} \frac{\abs{\bv}^2}{2}~ f(\bz)~ \intd\bz + \frac{1}{2} \int_{\RZ\times\RZ} f(\bz_1)~ \frac{\left(\bx_1-\bx_2\right)^2}{2}~ f(\bz_2)~ \intd(\bz_1,\bz_2),
\end{equation}
yielding the Vlasov equation
\begin{equation}
\tag{Vl} \label{eqn:harmonic:vl}
\partial_t f(t) = \symp{\frac{\abs{\bv}^2}{2} + \convol{\Gamma}{f(t)}(\bx)}{f(t)} = - \bv\cdot\naX f(t) + \convol{\nabla\Gamma}{f(t)}\cdot \naV f(t).
\end{equation}
While the condition of a fixed center of mass at $\bx=\bn$ automatically yields $\convol{\nabla\Gamma}{f(t)}(\bx)=\bx$, the existence of a large variety of periodic solutions is not at all surprising for this model. 

\begin{prop}[Solutions of Harmonic Vlasov]
\label{prop:harmonic:solutions}
Let $\mathring{f}: \RX\times\RV\to\RNp$ be a differentiable function, s.t.
\begin{equation*}
\int \mathring{f}(\bx,\bv)~ \intd(\bx,\bv) = 1 \quad \text{and} \quad \int\left(\abs{\bx}+\abs{\bv}\right) \mathring{f}(\bx,\bv)~ \intd(\bx,\bv) < \infty.
\end{equation*}
Then up to translation $\mathring{f}$ gives rise to a solution with period $2\pi$. This period is not necessarily minimal.
\end{prop}

\begin{proof}
W.l.o.g. assume (by appropriate translation) that
\begin{equation*}
\int \twovec{\bx}{\bv} \mathring{f}(\bx,\bv)~ \intd(\bx,\bv) = \twovec{\bn}{\bn}.
\end{equation*}
One easily verifies that
\begin{equation*}
f(t,\bx,\bv) \equiv \mathring{f}\left(\bx \cos t - \bv \sin t, \bx \sin t + \bv \cos t\right)
\end{equation*}
is the corresponding solution of \eqref{eqn:harmonic:vl}. Adding the center of mass motion for general initial data completes the proof.
\end{proof}

\noindent The improvement achieved by the symmetry reduction method in this paper is to algebraically characterize the \textit{minimal} period and stationary solutions, respectively. Nevertheless, the focus really lies on the method itself and the helpful insights it gives in order to solve more complicated systems.

\subsection{Phase equivariant reduction and transformations}
The associated Vlasov Hamiltonian of the energy functional \eqref{eqn:harmonic:vl-energy} constructed from \eqref{eqn:harmonic:vlasov-hamiltonian} is
\begin{equation*}
\HVF(\alpha) = \frac{1}{2}\SP{\alpha}{\bv \cdot \frac{1}{\mi}~\naX \alpha} + \Im \SP{\naV \alpha}{\convol{\Gamma}{\abs{\alpha}^2}~ \naX\alpha}.
\end{equation*}
In this particular case proves reasonable to apply a partial Fourier transform in the velocity coordinates $\bv$, i.e.,
\begin{equation*}
\fourier: \Lp{2}{\bz} \to \Lp{2}{\VF{\bz}}, \quad \alpha \mapsto \VF{\alpha}(\VF{\bz}) \equiv \VF{\alpha}(\bx,\xi) \equiv \frac{1}{\left(2\pi\right)^{\frac{\ddim}{2}}} \int_{\RV} \alpha(\bx,\bv)~ e^{-\mi \bv\cdot\xi}~ \intd\bv.
\end{equation*}
This transformation, already fruitfully used in \cite{neisspickl}, is of course an isometry of $\Lp{2}{}$ and thereby also a symplectic diffeomorphism. The Hamiltonian formalism therefore is protected under this transformation. It yields the \textbf{Pseudo Hartree Hamiltonian}
\begin{align*}
\HHF(\VF{\alpha}) = \frac{1}{2} \SP{\VF{\alpha}}{\left(\naX\cdot\naXi + \frac{1}{2} \convol{\VF{V}}{\abs{\VF{\alpha}}^2}\right)\VF{\alpha}}, \quad \text{where} \quad \VF{V}(\bx,\xi) \equiv -\nabla\Gamma(\bx)\cdot\xi = -\bx\cdot\xi.
\end{align*}
In this particular case, a second transformation simplifies the structure significantly. In fact, the self-inverse (hence volume-preserving) phase space transformation
\begin{equation*}
\tau: \RZH \to \RW, \quad (\bx,\xi) \mapsto \bw \equiv \left(\bq,\bp\right) \equiv \left(\frac{1}{\sqrt{2}}(\bx+\xi), \frac{1}{\sqrt{2}} (\bx-\xi)\right), 
\end{equation*}
and the notation $\beta=\tau_*\VF{\alpha} \equiv \VF{\alpha}\circ\inv{\tau}$, $W(\bq,\bp)\equiv\VF{V}\circ\inv{\tau}(\bq,\bp) = -\frac{\abs{\bq}^2}{2} + \frac{\abs{\bp}^2}{2}$, result in the new Hamiltonian
\begin{align}
\label{eqn:harmonic:hamilton-functional}
\HHF(\beta) =&~ \frac{1}{2}\SP{\beta}{\left(\frac{1}{2} \LQ - \frac{1}{2} \LP + \frac{1}{2}\convol{W}{\abs{\beta}^2}\right)\beta} \\
\nonumber
=&~ \frac{1}{2} \SP{\beta} {\left(\frac{1}{2} \LQ - \frac{\abs{\bq}^2}{2} \Lpn{\beta}{2}{\bw}^2 - \frac{1}{2} \LP + \frac{\abs{\bp}^2}{2} \Lpn{\beta}{2}{\bw}^2\right)~ \beta} + \frac{1}{4} \left|\SP{\beta}{\bq~ \beta}\right|^2 - \frac{1}{4} \abs{\SP{\beta}{\bp~\beta}}^2.
\end{align}
Also $\tau_*$ is a symplectic diffeomorphism, because $\tau$ preserves volume. Ultimately, this yields the Hamiltonian equation of motion for $\beta$ with $\HamVec{\HHF}$ denoting the Hamiltonian vector field
\begin{align}
\nonumber
\HamVec{\HHF}\left(\beta\right) =&~ \frac{1}{\mi}\left(\frac{1}{2} \LQ - \frac{\abs{\bq}^2}{2} - \frac{1}{2} \LP + \frac{\abs{\bp}^2}{2} \right) \beta 
+\frac{1}{\mi}\left(\SP{\beta}{\bq~ \beta} \cdot \bq - \SP{\beta}{\bp~ \beta} \cdot \bp\right) \beta \\
\label{eqn:harmonic:unprojected-vf}
&+\frac{1}{\mi}\left(-\SP{\beta}{\frac{\abs{\bq}^2}{2}~ \beta} + \SP{\beta}{\frac{\abs{\bp}^2}{2}~ \beta}\right) \beta
+\frac{1}{\mi}\left(\Lpn{\beta}{2}{\bw}^2-1\right)\frac{-\abs{\bq}^2+\abs{\bp}^2} {2}~ \beta.
\end{align}
From this representation one sees that the restriction to $\Sph{\Lp{2}{\bw}}$ and projection $\Pi$ from Proposition \ref{prop:harmonic:pe-symplecticity} greatly simplify the Hamiltonian vector field, since the restriction cancels the last summand and all terms in $\mi\RN\beta$ lie in the kernel of the projection. The \textbf{Harmonic Hartree equation} is
\begin{equation}
\tag{HHt} \label{eqn:harmonic:harmonic-hartree}
\partial_t\beta = \HamVec{\HHF}(\beta), \quad \beta: \RN \to \Lp{2}{\bw}.
\end{equation} \\

\noindent It turns out extremely useful to exploit the structural equivalence to the quantum mechanical harmonic oscillator and rewrite the equation with the algebraic ladder operators, a well-known formalism from textbook quantum mechanics \cite[Ch.3.1]{schwabl}, i.e.,
\begin{equation*}
\qma_i = \frac{1}{\sqrt{2}}\left(q_i+\partial_{q_i}\right), \quad \qma_i^* = \frac{1}{\sqrt{2}}\left(q_i-\partial_{q_i}\right), \quad \qmb_i = \frac{1}{\sqrt{2}} \left(p_i + \partial_{p_i}\right), \quad \qmb_i^* = \frac{1}{\sqrt{2}}\left(p_i-\partial_{p_i}\right),
\end{equation*}
along with the complete basis of eigenfunctions $\left\{\ket{\ba,\bb} : \ba,\bb\in\NN_0^\ddim\right\}$, that simultaneously diagonalize the commuting counting operators $\qma_i^*\qma_i$ and $\qmb_i^*\qmb_i$. We recall that this implies
\begin{equation*}
\qma_i~\ket{\ba,\bb} = \sqrt{a_i}~ \ket{\ba-\be_i,\bb}, \quad
\qma_i^*~\ket{\ba,\bb} = \sqrt{a_i+1}~ \ket{\ba+\be_i,\bb}, \quad
\qma_i^*\qma_i~ \ket{\ba,\bb} = a_i~ \ket{\ba,\bb},
\end{equation*}
and for $\qmb$ likewise. Due to its multiple occurence, it is also useful to define the \textit{linear excitation operator}
\begin{equation*}
\qmN \equiv \sum_{i=1}^{\ddim} \left(\qmb_i^*\qmb_i - \qma_i^*\qma_i\right).
\end{equation*}
For further convenience we define the simultaneous eigenspaces of $\sum_i\qma_i^*\qma_i$ and $\sum_i\qmb_i^*\qmb_i$:
\begin{equation*}
\fock{A,B} \equiv \vspanC{\ket{\ba,\bb}}{\sum_{i=1}^{\ddim} a_i = A, \sum_{i=1}^{\ddim} b_i=B}, \quad \fock{N} \equiv \bigoplus_{B-A=N} \fock{A,B}, \quad \Lp{2}{\bw} = \bigoplus_{N=0}^\infty \fock{N}.
\end{equation*} 
Especially, $\qmN$ is diagonalizable with
\begin{equation*}
\forall N\in\IN: \ker\left(\qmN-N\right) = \fock{N} \quad \text{and} \quad \Spec{\qmN} = \IN.
\end{equation*}
\\

\noindent In this algebraic notation, the Hamiltonian vector field can be represented upon replacing $q_i=\frac{1}{\sqrt{2}} (\qma_i+\qma_i^*)$, $p_i=\frac{1}{\sqrt{2}} (\qmb_i+\qmb_i^*)$ in \eqref{eqn:harmonic:unprojected-vf}
\begin{align}
\nonumber
\HamVec{\HHF}(\beta)
=&~ \frac{1}{\mi} \left(\qmN + \frac{1}{2} \SP{\beta}{\qmN\beta}\right)\beta 
+ \frac{1}{2\mi}\Re\SP{\beta}{\sum_{i=1}^{\ddim} \left(\qmb_i^*\qmb_i^* - \qma_i\qma_i\right) \beta} \beta \\
\nonumber
&- \frac{1}{\mi}\sum_{i=1}^{\ddim} \left(\Re \SP{\beta}{\qmb_i \beta} \left(\qmb_i+\qmb_i^*\right) - \Re \SP{\beta}{\qma_i \beta} \left(\qma_i+\qma_i^*\right)\right) \beta \\
\label{eqn:harmonic:algebraic-vf}
&+ \frac{1}{4\mi}\left(\Lpn{\beta}{2}{}^2 - 1\right) \left(2\qmN + \sum_{i=1}^{\ddim} \left(\qmb_i\qmb_i+\qmb_i^*\qmb_i^* - \qma_i\qma_i - \qma_i^*\qma_i^*\right)\right) \beta.
\end{align}
This is a well-defined $\Lp{2}{\bw}$ valued vector-field on the dense domain
\begin{equation}
\label{eqn:harmonic:vf-domain}
\domN \equiv \left\{\beta=\sum_{A,B=0}^{\infty} \beta_{A,B} \in\Lp{2}{\bw}: \beta_{A,B}\in\fock{A,B}, \sum_{A,B=0}^{\infty} (A^2+B^2) \Lpn{\beta_{A,B}}{2}{\bw}^2 < \infty\right\} \subset \Lp{2}{\bw}.
\end{equation} \\

\noindent With these bosonic creation and annihilation operators on $\Lp{2}{\bw}$, application of Proposition \ref{prop:harmonic:pe-symplecticity} provides at $\beta\in\Sph{\Lp{2}{\bw}}\cap\domN$ the Hamiltonian $\PHHF$ and the Hamiltonian vector field $\HamVec{\PHHF}$ 
\begin{align}
\label{eqn:harmonic:pe-harmonic-hamiltonian}
\PHHF(\Pi\beta) =&~ \HHF(\beta) \stackrel{\Lpn{\beta}{2}{\bw}=1}{=} \frac{1}{2} \SP{\beta} {\qmN \beta} + \frac{1}{2} \sum_{i=1}^{\ddim} \left|\Re\SP{\beta}{\qma_i \beta}\right|^2 - \frac{1}{2} \sum_{i=1}^{\ddim} \abs{\Re\SP{\beta}{\qmb_i~\beta}}^2, \\
\nonumber
\HamVec{\PHHF}(\Pi\beta) =&~ 
\intd\Pi_\beta\left(\HamVec{\HHF}(\beta)\right) \\
\nonumber
=&~ \Sph{1}\left(\beta,\frac{1}{\mi}\left(\qmN-\SP{\beta}{\qmN\beta}\right)\beta 
- \frac{1}{\mi}\sum_{i=1}^{\ddim} \left(\Re\SP{\beta}{\qmb_i \beta} \left(\qmb_i^* + \qmb_i\right) - 2\abs{\Re\SP{\beta}{\qmb_i\beta}}^2\right) \beta\right. \\
\label{eqn:harmonic:pe-harmonic-hamvec}
&\quad\quad \left.+ \frac{1}{\mi}\sum_{i=1}^{\ddim} \left(\Re\SP{\beta}{\qma_i \beta} \left(\qma_i^*+\qma_i\right) - 2\abs{\Re\SP{\beta}{\qma_i \beta}}^2\right) \beta\right).
\end{align}
The phase equivariant Hamiltonian Hartree equation is denoted by
\begin{equation}
\tag{HHt\textsuperscript{red}} \label{eqn:harmonic:pe-hamilton-hartree}
\partial_t \Pi\beta(t) = \HamVec{\PHHF}(\Pi\beta(t)), \quad \Pi\beta: \RN \to \Sph{\Lp{2}{\bw}}/\Sph{1}.
\end{equation}

\subsection{Stationary points and spectral classification}

\label{sec:harmonic:spectrum}

The easiest way to find periodic solutions is by bifurcation around stationary points of the Hamiltonian dynamics. Many classical results exist using this principle. We want to compute relative equilibria $\Pi\beta\in\Sph{\Lp{2}{\bw}}/\Sph{1}$ and identify possible periods of bifurcating families from the spectrum of 
\begin{equation*}
\Del{}\HamVec{\PHHF}(\Pi\beta): T_{\Pi\beta}\left(\Sph{\Lp{2}{\bw}}/\Sph{1}\right) \to T_{\Pi\beta}\left(\Sph{\Lp{2}{\bw}}/\Sph{1}\right).
\end{equation*}
Indeed, in this highly symmetric model, every eigenvector of the differential at a stationary point can be linked to a bifurcating periodic family or to a continuous symmetry, as will be outlined at the end of this subsection.

\begin{lem}[Relative equilibria]
\label{lem:harmonic:stationary-points}
The phase equivariant projection $\Pi\beta$ of any unit eigenvector $\beta\in\domN$ from \eqref{eqn:harmonic:vf-domain} of $\qmN$ is a stationary point of \eqref{eqn:harmonic:pe-hamilton-hartree}.
\end{lem}

\begin{proof}
This is immediate from the representation \eqref{eqn:harmonic:pe-harmonic-hamvec}, because $\Re\SP{\beta}{\qma_i\beta} = \Re\SP{\beta}{\qmb_i\beta} = 0$, as the operators $\qma_i$ ($\qmb_i$) will map $\beta$ to the eigenspace of the next higher (lower) eigenvalue of $\qmN$, as seen in the stated properties of the previous section, i.e.,
\begin{equation*}
\forall N\in\IN: \quad \qma_i(\fock{N}) \subseteq \fock{N+1}, \quad \qmb_i(\fock{N}) \subseteq \fock{N-1},
\end{equation*}
and those eigenspaces are orthogonal to one another.
\end{proof}

\begin{prop}[Spectral properties]
\label{prop:harmonic:spectral-properties}
For $N\in\IN$ at any unit eigenvector $\mathring{\beta}\in\ker\left(\qmN-N\right)\cap\Sph{\Lp{2}{\bw}}$ of $\qmN$, on the subspace
\begin{align*}
\perB{\Pi\mathring{\beta}} \equiv&~ \Sph{1}\left\{\begin{array}{c}
(\mathring{\beta},\delta\beta): \delta\beta\in\domN, \forall 1\leq i\leq\ddim: \\ \Re\SP{\delta\beta}{\qma_i\mathring{\beta}} = \Re\SP{\delta\beta}{\qma_i^*\mathring{\beta}} = 0, \\
\Re\SP{\delta\beta}{\qmb_i\mathring{\beta}} = \Re\SP{\delta\beta}{\qmb_i^*\mathring{\beta}} = 0
\end{array}\right\} 
\subseteq \tang{\Pi\mathring{\beta}}{\left(\Sph{\Lp{2}{\bw}}/\Sph{1}\right)}
\end{align*}
which has real codimension $4\ddim$ in  $\tang{\Pi\beta}{\left(\Sph{\Lp{2}{\bw}}/\Sph{1}\right)}$ one finds \begin{equation*}
\left.\Del{}\HamVec{\PHHF}(\Pi\mathring{\beta})\right|_{\perB{\Pi\mathring{\beta}}} \left(\Sph{1}(\mathring{\beta},\delta\beta)\right) = \Sph{1}\left(\mathring{\beta},\frac{1}{\mi} (\qmN-N)\delta\beta\right) \quad \text{and} \quad \mi\IN \subseteq \Spec{\Del{}\HamVec{\PHHF}(\Pi\mathring{\beta})}.
\end{equation*}
\end{prop}

\begin{proof}
Carrying out a spectral analysis on the quotient manifold $\Sph{\Lp{2}{\bw}}/\Sph{1}$ is possible in the chart of Lemma \ref{lem:harmonic:coordinate-representations}-(ii). In this chart $\HamVec{\PHHF}$ is represented by the vector field on $\beta\in\domN\cap\Sph{\Lp{2}{\bw}}$
\begin{align*}
\XV(\beta) =&~ \frac{1}{\mi}\left(\qmN-\SP{\beta}{\qmN\beta}\right)\beta 
- \frac{1}{\mi}\sum_{i=1}^{\ddim} \left(\Re\SP{\beta}{\qmb_i \beta} \left(\qmb_i^* + \qmb_i\right) - 2\abs{\Re\SP{\beta}{\qmb_i\beta}}^2\right) \beta \\
&~ + \frac{1}{\mi}\sum_{i=1}^{\ddim} \left(\Re\SP{\beta}{\qma_i \beta} \left(\qma_i^*+\qma_i\right) - 2\abs{\Re\SP{\beta}{\qma_i \beta}}^2\right) \beta,
\end{align*}
s.t. $\SP{\beta}{\XV(\beta)}=0$. It now suffices to compute the spectrum of $\Del{}\XV(\mathring{\beta}): \{\mathring{\beta}\}^\perp \to \{\mathring{\beta}\}^\perp$, because its spectrum equals the one of $\Del{}\HamVec{\PHHF}$ and their eigenspaces map to one another.

For the linearization of $\XV$ at its zero $\mathring{\beta}$ with $\qmN\mathring{\beta}=N\mathring{\beta}$, $\SP{\mathring{\beta}}{\delta\beta}=0$, one computes
\begin{align*}
\Del{}\XV(\mathring{\beta})(\delta\beta) =&~ 
\frac{1}{\mi}\left(\qmN-N\right)\delta\beta
- \frac{1}{\mi} \sum_{i=1}^{\ddim} \Re\SP{\delta\beta}{\left(\qmb_i^*+\qmb_i\right) \mathring{\beta}} \left(\qmb_i^* + \qmb_i\right) \mathring{\beta} \\
&~ + \frac{1}{\mi} \sum_{i=1}^{\ddim} \Re\SP{\delta\beta}{\left(\qma_i^* + \qma_i\right) \mathring{\beta}} \left(\qma_i^*+\qma_i\right) \mathring{\beta}.
\end{align*}
The two sums combined are a nuclear perturbation of the operator $\frac{1}{\mi}(\qmN-N)$ with finite-dimensional range. Indeed,
\begin{equation*}
\perB{\mathring{\beta}} \equiv \left\{\delta\beta\in\domN: \SP{\delta\beta}{\mathring{\beta}} = 0, \forall 1\leq i\leq\ddim: 
\begin{array}{c}
\Re\SP{\delta\beta}{\qma_i\mathring{\beta}} = \Re\SP{\delta\beta}{\qma_i^*\mathring{\beta}} = 0, \\
\Re\SP{\delta\beta}{\qmb_i\mathring{\beta}} = \Re\SP{\delta\beta}{\qmb_i^*\mathring{\beta}} = 0
\end{array}\right\}
\end{equation*}
lies in the kernel of the perturbation, proving $\left.\Del{}\XV(\mathring{\beta})\right|_{\perB{\mathring{\beta}}} = \frac{1}{\mi}(\qmN-N)$. Hence, every imaginary integer is an infinitely degenerate eigenvalue of $\Del{}\XV(\mathring{\beta})$. As this is a computation in a chart, this result transfers to $\left.\Del{}\HamVec{\PHHF}(\Pi\mathring{\beta})\right|_{\perB{\Pi\mathring{\beta}}}$, given $\perB{\Pi\mathring{\beta}}=\Pi(\mathring{\beta},\perB{\mathring{\beta}})$ respectively.
\end{proof}

\noindent The spectral decomposition hints at candidates for periodic solutions around the stationary point $\Pi\mathring{\beta}$. Nevertheless, in this highly degenerate problem, periodic families also bifurcate in directions outside eigenspaces\footnote{Which corresponds to oscillation index $>2$, see Section \ref{sec:harmonic:algebraic}}, as can be seen in Theorem \ref{thm:harmonic:topological-version}.

However, one derives a satisfying spectral characterization of periodic families as the $4\ddim$-dimensional complement of $\perB{\Pi\mathring{\beta}}$ in Theorem \ref{thm:harmonic:spectral-parameterization} reflects the kernel generated by the unreduced symmetry group $\RX\times\RXi$ from Proposition \ref{prop:harmonic:noether}, as for example,
\begin{equation*}
0 = \frac{1}{\sqrt{2}} \Re\SP{\delta\beta}{\left(\qma_i - \qma_i^*\right)\mathring{\beta}} = \Re\SP{\delta\beta}{\partial_{q_i}\mathring{\beta}},
\end{equation*}
showing orthogonality to the local generator of translation in $\bq$. \\

\noindent The following result proves that all spectrally admitted bifurcating periodic solutions around the identified relative equilibria actually exist. Although this problem's highly degenerated spectrum inhibits to obtain these solutions implicitly from bifurcation theory, we at least obtain a full spectral mapping, justified from explicit computations in the next section.

\begin{thm}[Spectral parameterization of periodic families]
\label{thm:harmonic:spectral-parameterization}
Let $\mathring{\beta}\in\domN\cap\Sph{\Lp{2}{\bz}}$ be a normed eigenvector of $\qmN$ and choose the subspace $\perB{\Pi\mathring{\beta}}\subseteq T_{\Pi\mathring{\beta}}\left(\Sph{\Lp{2}{\bz}}/\Sph{1}\right)$ as in Proposition \ref{prop:harmonic:spectral-properties}. Then there is a canonical injection
\begin{equation*}
\begin{array}{ccc}
\left\{\begin{array}{c}
(\Sph{1}\tilde{\beta},L): \tilde{\beta} \in \perB{\Pi\mathring{\beta}}, \\
\Lpn{\tilde{\beta}}{2}{\bw}=1, L \in \CN, \\
\tilde{\beta} \in \ker \left(\Del{}\HamVec{\PHHF}(\Pi\mathring{\beta}) - \mi L\right)
\end{array}\right\} & \to & \left\{\begin{array}{c}
\text{Continuous families of periodic solutions} \\
\text{and their periods bifurcating from $\Pi\mathring{\beta}$}
\end{array}\right\} \\
\left(\Sph{1}\tilde{\beta}, L\right) & \mapsto & \left(\begin{array}{ccc}
[0,\pi] & \to & \left\{\text{Closed curves in $\Sph{\Lp{2}{}}/\Sph{1}$}\right\} \times \RNp \\
\gamma & \mapsto & \left(\Sph{1}\left(\cos\frac{\gamma}{2}~ \mathring{\beta} + \sin\frac{\gamma}{2}~ \Sph{1}\tilde{\beta}\right), \frac{2\pi}{L}\right)
\end{array}\right).
\end{array}
\end{equation*}
\end{thm}

\begin{proof}
This is a consequence of the explicit computations of the next section. It is immediately implied by Corollary \ref{cor:harmonic:interpolating-families} and illustrated in Figure \ref{fig:harmonic:phase-portrait}.
\end{proof}

\subsection{Algebraic computation of periodic solutions}

\label{sec:harmonic:algebraic}

\noindent The highly resonant spectrum $\mi\IN\subseteq\Spec{\Del{}\HamVec{\PHHF}(\Pi\beta)}$ at the listed relative equilibria of Section \ref{sec:harmonic:spectrum} is a strong obstacle to apply classical bifurcation theory to prove the existence of periodic families around them. However, the system fortunately admits explicit computation of periodic solutions, because it possesses a lot of \textit{finite-dimensional} subspheres in $\Sph{\Lp{2}{\bw}}$ which are left invariant by the dynamics. They are the key to computing a rich variety of periodic solutions. The spheres are constructed as intersections of $\Sph{\Lp{2}{\bw}}$ with finite-dimensional complex subspaces of $\Lp{2}{\bw}$.

\begin{defn}[Centered subspaces]
\label{defn:harmonic:centered-subspaces}
Let $W\subset\Lp{2}{\bw}$ be a finite-dimensional complex subspace. $W$ is called \textbf{centered} if
\begin{enumerate}[label=(\roman*)]
\item $W = \bigoplus_{N\in\IN} W_N$, where $W_N \subseteq \fock{N} = \ker\left(\qmN-N\right)$, (or equivalently $\qmN(W)\subseteq W$),
\item $\forall \beta\in W$, $\forall 1\leq i\leq\ddim$: $\Re\SP {\beta}{\qma_i\beta} = \Re\SP{\beta}{\qmb_i\beta} = 0$.
\end{enumerate}
The number of non-zero subspaces $W_N$ in the decomposition will be called \textbf{decomposition index} of $W$, i.e.,
\begin{equation*}
\InDec{W} \equiv \#\left\{N\in\IN: \dim_\CN W_N > 0\right\} \in \NN_0\cup \{\infty\}.
\end{equation*}
\end{defn}

\begin{exam}
\label{ex:harmonic:centered-subspaces}
\textbf{(i).} Any $W=\bigoplus_{N\in\IN} W_N$, s.t. $W_N, W_M\neq \{\bn\} \Rightarrow \abs{N-M} \neq 1$, is centered. For $\beta=\sum_N \beta_N$, one computes using $\qma_i(\fock{N})\subseteq\fock{N+1}$ and orthogonality of the $\fock{N}$
\begin{equation*}
\Re\SP{\beta}{\qma_i \beta} = \sum_{M,N\in\IN} \Re\SP{\beta_N}{\qma_i\beta_M} = \sum_{N\in\IN} \Re\SP{\beta_N}{\qma_i\beta_{N-1}} = 0.
\end{equation*} \\

\noindent \textbf{(ii).} If $\beta=\sum_{N\in\IN} \beta_N$ is contained in some centered subspace $W=\bigoplus_{N\in\IN} W_N$, then necessarily $\CN\beta_N\subseteq W_N$ by the vector space property of the $W_N$, hence $\bigoplus_{N\in\IN} \CN\beta_N\subseteq W$. Also $\bigoplus_{N\in\IN} \CN\beta_N$ obviously complies with the decomposition condition. The second condition is clearly true for any subset of $W$. Therefore, $\bigoplus_{N\in\IN} \CN\beta_N$ is the minimal centered subspace containing $\beta$.
\end{exam}

\noindent Although it seems a little technical at first, condition \ref{defn:harmonic:centered-subspaces}-(ii) is actually very natural. In fact,
\begin{equation*}
\Re\SP{\beta}{\qma_i\beta} = \frac{1}{\sqrt{2}} \SP{\beta}{q_i~\beta} \quad \text{and} \quad \Re\SP{\beta}{\qmb_i\beta} = \frac{1}{\sqrt{2}} \SP{\beta}{p_i~\beta},
\end{equation*}
therefore this condition simply allows to cancel the translation invariance in both $(\bq,\bp)$ and centering them at $(\bn,\bn)$.

One quickly checks that the family of centered subspaces is $\cap$-stable. This admits the following definition.

\begin{defn}[Oscillation index]
\label{defn:harmonic:oscillation-index}
Let $\beta\in\Sph{\Lp{2}{\bw}}$ be a function. The \textbf{oscillation index} of $\beta$ is
\begin{equation*}
\InOsc{\beta} \equiv \inf\left\{\InDec{W}: W\subset\Lp{2}{\bw}~ \text{centered},~ \beta\in W\right\} \in \NN_0\cup \{\infty\}.
\end{equation*}
\end{defn}

\noindent The following Lemma shows that the centered subspaces actually encode some \textit{hidden} conservation laws, because their unit spheres are stable under the dynamics.

\begin{lem}[$\Sph{\Lp{2}{}}$ vector field]
\label{lem:harmonic:vector-field}
Let $\YV: \domN \to \Lp{2}{\bw}$ be the vector field defined by
\begin{align*}
\YV(\beta)
=&~ \frac{1}{\mi} \sum_{i=1}^{\ddim} \left(\qmb_i^* \qmb_i + \frac{1}{2}\Re\SP{\beta}{\left(\qmb_i\qmb_i+\qmb_i^*\qmb_i\right) \beta} - \Re\SP{\beta}{\qmb_i \beta} \left(\qmb_i+\qmb_i^*\right)\right) \beta \\
&- \frac{1}{\mi} \sum_{i=1}^{\ddim} \left(\qma_i^* \qma_i + \frac{1}{2} \Re\SP{\beta}{\left(\qma_i\qma_i+\qma_i^*\qma_i\right) \beta} - \Re\SP{\beta}{\qma_i \beta} \left(\qma_i+\qma_i^*\right)\right) \beta.
\end{align*}
Then the following claims hold
\begin{enumerate}[label=(\roman*)]
\item $\YV \equiv \HamVec{\HHF}$ on $\domN\cap\Sph{\Lp{2}{\bw}}$, hence $\YV: \domN\cap\Sph{\Lp{2}{\bw}} \to T\Sph{\Lp{2}{\bw}}$ is well-defined, and
\item for all centered subspaces $W\subset\Lp{2}{\bw}$ we have $\YV(W)\subseteq W$. In particular, $\YV : \Sph{W} \to T\Sph{W}$ is a smooth vector field.
\end{enumerate}
\end{lem}

\begin{proof}
\textbf{(i).} Note that $\YV$ equals $\HamVec{\HHF}$ from \eqref{eqn:harmonic:algebraic-vf} up to a term containing $\left(\Lpn{\beta}{2}{\bw}^2-1\right)$ in the product. Thus on $\Sph{\Lp{2}{\bw}}$ they are the same. \\

\noindent \textbf{(ii).} Let $W\subset\Lp{2}{\bw}$ be a finite-dimensional centered subspace. Pick any $\beta = \sum_N \beta_N \in W$ by the decomposition of $W$, in particular just finitely many $\beta_N\neq 0$. We then compute by the condition (ii) of Definition \ref{defn:harmonic:centered-subspaces}
\begin{align*}
\YV(\beta) =&~ \frac{1}{\mi} \sum_{N\in\IN} \left(\qmN + \frac{1}{2} \SP {\beta} {\qmN\beta}\right)\beta_N + \frac{1}{\mi} \left(\frac{1}{2} \Re\SP{\beta} {\sum_{i=1}^{\ddim}\left(\qmb_i^*\qmb_i^* - \qma_i\qma_i\right) \beta}\right) \beta \\
=&~ \frac{1}{\mi} \sum_{N\in\IN} \left(N + \frac{1}{2} \SP{\beta}{\qmN\beta} + \frac{1}{2} \Re \sum_{M\in\IN} \SP{\beta_M}{\left(\sum_{i=1}^{\ddim} \left(\qmb_i^*\qmb_i^* - \qma_i\qma_i\right)\right)\beta_{M-2}}\right) \beta_N \\
\in&~ \sum_{N\in\IN} \mi\RN\beta_N \subseteq W.
\end{align*}
As the vector field is in fact only a polynomial in the (finitely many) coefficients $\beta_N$, its smoothness restricted on $W$ is immediate.
\end{proof}

\noindent In order to learn more about the value of the symmetry reduction, we compute solutions not only in the quotient manifold $\Sph{\Lp{2}{}}/\Sph{1}$ but also in the sphere $\Sph{\Lp{2}{}}$. 

The given example of Harmonic Vlasov already proves that many trajectories are only closed by passing on to the quotient, highlighting the value of the symmetry reduction. An inspection of the next proof even shows that the global phase oscillation is not generally separable by an $e^{\mi \lambda t}$ ansatz, emphasizing the necessity of the reduction method over this ansatz.

We remind the reader that \textit{relatively periodic} means periodic in the quotient manifold.

\begin{thm}[Relatively periodic families of Harmonic Hartree]
\label{thm:harmonic:periodic-families}
Let $\mathring{\beta}\in\Sph{\Lp{2}{\bw}}$ have a finite oscillation index $\InOsc{\beta}=J\in\NN$ and be contained in the minimal centered subspace $W=\oplus_N W_N$. Then $\mathring{\beta}$ gives rise to a global solution of \eqref{eqn:harmonic:harmonic-hartree} which always remains in $\Sph{W}$, is relatively periodic for $J>1$, and relatively constant for $J=1$.

In the case of $J\geq2$ the relative period is given by
\begin{align*}
\frac{2\pi} 
{\gcd\left(\left\{\abs{N-M} : \dim_\CN W_M, \dim_\CN W_N >0\right\}\right)}.
\end{align*}
In particular, the solution is classically periodic if and only if
\begin{equation*}
\sum_{i=1}^{\ddim} \SP{\beta} {\left(\qmb_i^*\qmb_i-\qma_i^*\qma_i\right) \beta} = \SP{\beta}{\qmN\beta} \in \QN.
\end{equation*}
The classical period is a multiple of the relative period.
\end{thm}

\begin{proof}
Let $\mathring{\beta} = \sum_N \mathring{\beta}_N$ be a decomposition according to the minimal centered space $W$ containing $\mathring{\beta}$. Let $\mathcal{N}\equiv \left\{N\in\IN: \dim_\CN W_N >0\right\}$ denote the finite set of indices with non-trivial $W_N$. By Lemma \ref{lem:harmonic:vector-field}, the initial value problem
\begin{align*}
\partial_t\beta(t) = \YV\left(\beta(t)\right), \quad \beta(0) = \mathring{\beta} \in \Sph{W},
\end{align*}
is well-posed on the finite-dimensional compact manifold $\Sph{W}\simeq \Sph{2J-1}\subset \CN^J$ and possesses by the smoothness and boundedness of $\YV$ on $\Sph{W}$ a global smooth solution. By Lemma \ref{lem:harmonic:vector-field}-(i), this is actually a solution to \eqref{eqn:harmonic:harmonic-hartree}.

As seen in Example \ref{ex:harmonic:centered-subspaces}-(ii), minimality of $W$ already implies the strict form $W_N=\CN\mathring{\beta}_N$. Hence, one rewrites the equations of motion in the orthogonal decomposition
\begin{align}
\nonumber
\forall N\in\mathcal{N}: \quad \mi~ \dot{\beta}_N(t) =&~ \left(N + \frac{1}{2} \SP{\beta(t)}{\qmN\beta(t)}\right) \beta_N(t) \\
\label{eqn:harmonic:coordinate-hht}
&+ \left(\frac{1}{2} \Re \sum_{M,M-2\in\mathcal{N}} \SP{\beta_M(t)}{\left(\sum_{i=1}^{\ddim} \left(\qmb_i^*\qmb_i^* - \qma_i\qma_i\right)\right)\beta_{M-2}(t)} \right) \beta_N(t).
\end{align}
Now by \eqref{eqn:harmonic:coordinate-hht} we have $\dot{\beta}_N \in \mi\RN \beta_N$, thus the $\Lp{2}{}$ norms of all $\beta_N$ must be conserved and they are $\Lpn{\mathring{\beta}_N}{2}{\bw}$ for $t=0$. In particular, there is only phase oscillation inside $W_N$. Consequently, $t\mapsto \SP{\beta(t)}{\qmN\beta(t)} \equiv \langle\qmN\rangle$ is constant. We therefore choose the ansatz
\begin{equation*}
\forall N\in\mathcal{N}: \quad \beta_N(t) = \mathring{\beta}_N~ \exp\left(-\mi\left(N+\frac{1}{2}\langle\qmN\rangle\right) t - \mi\varphi(t)\right), \quad \varphi: \RN \to \RN, \quad \varphi(0) = 0,
\end{equation*}
that reduces the full system of ODEs to the single ODE
\begin{equation*}
\dot{\varphi}(t) = \frac{1}{2} \Re e^{2\mi t} \sum_{M\in\IN} \SP{\mathring{\beta}_M}{\left(\sum_{i=1}^{\ddim} \left(\qmb_i^*\qmb_i^* - \qma_i\qma_i\right)\right)\mathring{\beta}_{M-2}}, \quad \varphi(0) = 0,
\end{equation*}
which possesses a unique global solution which is given by a linear combination of $\sin(2t)$ and $\cos(2t)$.

We remark that the common phase $\frac{1}{2}\langle\mathfrak{N}\rangle t+\varphi(t)$ is canceled by the $\Sph{1}$ modulo operation. The minimal common relative period of the $t\mapsto\beta_N(t)$ is now found by the formula given in the theorem. In the special case of $J=1$, the dynamics is only phase oscillation and therefore relatively constant.

Given that $\varphi$ has period $\pi$, $\langle\qmN\rangle\in\QN$ is equivalent to the solution being classically periodic.
\end{proof}

\begin{thm}[Topological characterization of periodic orbits]
\label{thm:harmonic:topological-version}
Let $W\subset\Lp{2}{\bw}$ be a centered subspace of finite dimension $I = \dim_\CN W \in\NN$. The vector field
\begin{equation*}
\YV: \Sph{W} \to T\Sph{W}
\end{equation*}
from Lemma \ref{lem:harmonic:vector-field} is equivariant under $\Sph{1}$ action and therefore can be pushed forward under the projection
\begin{equation*}
\Pi: \Sph{W} \to \Sph{W}/\Sph{1} \simeq \CN P^{I-1}.
\end{equation*}
The projection of each trajectory of \eqref{eqn:harmonic:harmonic-hartree} through $\beta\in\Sph{W}$ onto $\Sph{W}/\Sph{1}$ is closed and has constant velocity
\begin{equation*}
v\left(\beta\right)
\equiv \sqrt{\SP{\beta}{\qmN^2\beta} - \SP{\beta}{\qmN\beta}^2}
\end{equation*}
The projections of two trajectories either coincide or are disjoint.
\end{thm}

\begin{proof}
Let $W\subset\Lp{2}{\bw}$ be a centered subspace. The $\Sph{1}$ equivariancy of $\YV$ follows from
\begin{equation*}
\forall \zeta\in\Sph{1}~ \forall\beta\in\Sph{W}: \quad \inv{\zeta}~ \YV(\zeta~\beta) = \YV(\beta).
\end{equation*}
Let $\Pi_*\YV$ denote the push-forward. In particular, for solutions of \eqref{eqn:harmonic:harmonic-hartree}
\begin{equation*}
\partial_t\beta = \YV(\beta) \quad \Rightarrow \quad \partial_t\Pi\beta = \intd\Pi_\beta(\partial_t\beta) = \intd\Pi_\beta(\YV(\beta)) = \left(\Pi_*\YV\right)\left(\Pi\beta\right).
\end{equation*}
We conclude that trajectories in $\Sph{W}/\Sph{1}$ cannot intersect because they solve a first order autonomous ODE with smooth vector field. Their closedness was shown in Theorem \ref{thm:harmonic:periodic-families}. 

It remains to compute the velocity in the metric of the surrounding space. Let $\beta\in\Sph{W}$ be given. We can model the differential of $\Pi$ at $\beta $ by the charts of Lemma \ref{lem:harmonic:coordinate-representations}-(i) as
\begin{equation*}
\intd\Pi_\beta : \begin{array}{ccc}
T_{\beta}\Sph{W} = \left\{\gamma\in W: \Re\SP{\beta}{\gamma} = 0\right\} & \to & T_{\Pi\beta}\left(\Sph{W}/\Sph{1}\right) = \left\{\gamma\in W: \SP{\beta}{\gamma} = 0\right\} \\
\gamma & \mapsto & \gamma - \SP{\gamma}{\beta} \beta.
\end{array}
\end{equation*}
This yields for the velocity of a trajectory $t\mapsto \beta(t) \in\Sph{W}$ in $\Sph{W}/\Sph{1}$ the following expression:
\begin{align*}
\norm{\partial_t \Pi\beta}_{T_{\Pi\beta}\Sph{W}/\Sph{1}}^2 
=&~ \abs{\intd\Pi_\beta(\partial_t\beta)}_{T_{\Pi\beta}\Sph{W}/\Sph{1}}^2 
= \Lpn{\dot{\beta}-\SP{\dot{\beta}}{\beta} \beta}{2}{\bw}^2 
= \Lpn{\dot{\beta}}{2}{\bw}^2 - \abs{\SP{\dot{\beta}}{\beta}}^2 \\
=&~ \Lpn{\YV(\beta)}{2}{\bw}^2 - \abs{\SP{\YV(\beta)}{\beta}}^2 = (*).
\end{align*}
Using the representation of solutions $\beta=\sum_N \beta_N$ from Theorem \ref{thm:harmonic:periodic-families} and the representation of $\YV$ from the proof of Lemma \ref{lem:harmonic:vector-field}, one computes
\begin{align*}
(*)
=&~ \sum_{N\in\IN} \Lpn{\beta_N}{2}{\bw}^2 \left(N + \frac{1}{2} \SP{\beta}{\qmN\beta} + \frac{1}{2} \Re \sum_{M\in\IN} \SP{\beta_M}{\left(\sum_{i=1}^{\ddim} \left(\qmb_i^*\qmb_i^* - \qma_i\qma_i\right)\right)\beta_{M-2}} \right)^2 \\
&-~ \left(\sum_{N\in\IN} \Lpn{\beta_N}{2}{\bw}^2 \left(N + \frac{1}{2} \SP{\beta}{\qmN\beta} + \frac{1}{2} \Re \sum_{M\in\IN} \SP{\beta_M}{\left(\sum_{i=1}^{\ddim} \left(\qmb_i^*\qmb_i^* - \qma_i\qma_i\right)\right)\beta_{M-2}} \right)\right)^2 \\
=&~ \SP{\beta}{\qmN^2\beta} - \SP{\beta}{\qmN\beta}^2.
\end{align*}
\end{proof}

\begin{figure}[htp]
\centering

\begin{tikzpicture} 

\def\R{2.5} 
\def\angEl{25} 
\def\angBeta{-20}
\def\angGamma{70}
\pgfmathsetmacro\H{\R*cos(\angEl)} 
\filldraw[ball color=white] (0,0) circle (\R);
\foreach \t in {-80,-60,...,80} { \DrawLatitudeCircle[\R]{\t} }

\LongitudePlane[pzplane]{\angEl}{\angBeta}

\DrawLongitudeHalfCircle[\R]{\angBeta} 

\coordinate[mark coordinate] (N) at (0,\H);
\coordinate[mark coordinate] (S) at (0,-\H);
\path[pzplane] (90-\angGamma:\R) coordinate[mark coordinate] (P);
\coordinate (O) at (0,0);
\draw[->] (O) -- (P);
\draw[pzplane,->,thin] (90:0.5*\R) to[bend left=35]
    node[midway,above right] {$\gamma$} (90-\angGamma:0.5*\R);
\draw (0,-\H) -- (0,\H);

\node [above] at (0,1.1*\H) {$\Sph{1}\beta_N$};
\node [below] at (0,-1.1*\H) {$\Sph{1}\beta_M$};
\node [right] at (\R*1.1,\R*0.1) {$\Sph{1}\left(\cos\frac{\gamma}{2}~ \beta_N + \sin\frac{\gamma}{2}~ \beta_M\right)$};

\end{tikzpicture}
\caption{\label{fig:harmonic:phase-portrait} The phase portrait of interpolating families given in Corollary \ref{cor:harmonic:interpolating-families}. For representation issues a diffeomorphic map between $\Sph{W}/\Sph{1} \simeq \CN P^1 \simeq \Sph{2}\subset\RN^3$ was chosen. The north and south pole are relatively constant solutions, along the latitudes flow the relatively periodic solutions with $\Pi_*\YV$, whose trajectories can be parameterized by the angle $\gamma\in(0,\pi)$.}
\end{figure}
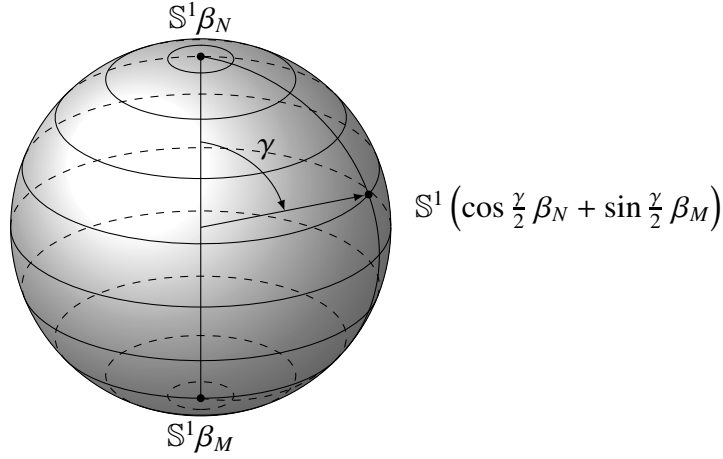

\begin{cor}[Interpolating families]
\label{cor:harmonic:interpolating-families}
Let $\beta_N,\beta_M\in\Sph{\Lp{2}{\bw}}$ be two eigenvectors of the linear excitation operator $\qmN$ with different eigenvalues $M\neq N$, s.t. in addition $W\equiv\CN\beta_M\oplus\CN\beta_N$ is centered. Then there exists a continuous family of relatively periodic solutions of the phase equivariant Harmonic Hartree \eqref{eqn:harmonic:pe-hamilton-hartree}, s.t. all trajectories have the same period.
\end{cor}

\begin{proof}
The space $W\equiv\CN\beta_N\oplus\CN\beta_M$ has complex dimension and decomposition index $2$. Therefore, all elements of $W$ which are not a multiple of the eigenvectors lead to relatively periodic solutions with period $\frac{2\pi}{\abs{M-N}}$ by Theorem \ref{thm:harmonic:periodic-families}. The continuous interpolation parameter can be an angle $\gamma$, as illustrated in Figure \ref{fig:harmonic:phase-portrait}.
\end{proof}

\noindent By the hierarchy of equations, we can transform any solution back to the $\Lp{1}{}$ picture of the classical Vlasov system. This leads to a rich variety of periodic solutions, previously difficult to classify.

\begin{thm}[Periodic solutions of Harmonic Vlasov]
Under the chain of transformations
\begin{equation*}
\Lp{2}{\bq,\bp} \stackrel{\tau^*}{\to} \Lp{2}{\bx,\xi} \stackrel{\text{F.T.}}{\to} \Lp{2}{\bx,\bv} \stackrel{\abs{\cdot}^2}{\to} \Lp{1}{\bx,\bv}
\end{equation*}
the set of initial conditions
\begin{equation*}
\Sph{\Lp{2}{\bq,\bp}}\cap \bigcup_{W~ \text{centered}} W
\end{equation*}
leads to non-trivial periodic solutions of the classic Harmonic Vlasov system \eqref{eqn:harmonic:vl}.
\end{thm}

\begin{proof}
Reverse transformation of the periodic curves obtained from Theorem \ref{thm:harmonic:periodic-families} yields the result immediately. In particular the relative periodicity turns into classic periodicity under the last transformation.
\end{proof}

\noindent Actually, many periodic solutions from the previous theorem can be computed very conveniently, if one chooses an initial state of the form
\begin{equation*}
\mathring{\beta} \propto \left(\text{polynomial in $\bq,\bp$}\right) \cdot \exp\left(-\frac{\abs{\bq}^2 + \abs{\bp}^2}{2}\right),
\end{equation*}
as is highlighted by the following example.

\begin{exam}
Choose the space dimension $\ddim=1$ and the two explicit states
\begin{align*}
\mathring{\beta}_{0,0}(q,p) \equiv \pi^{-\frac{1}{2}} e^{-\frac{q^2 + p^2}{2}}, \quad 
\mathring{\beta}_{2,0} \equiv \pi^{-\frac{1}{2}} e^{-\frac{q^2 + p^2}{2}} \frac{1}{\sqrt 2}\left(q^2 - 1\right).
\end{align*}
Using the angular parametrization consistent with Fig. \ref{fig:harmonic:phase-portrait}, we can compute for the time-dependent solutions
\begin{align*}
\Sph{1}\beta_\gamma(t,q,p) =&~ \Sph{1} \left(\cos\frac{\gamma}{2} \beta_{0,0}(q,p) + \sin\frac{\gamma}{2} e^{2\mi t}\beta_{2,0}(q,p)\right), \\
\Sph{1}\VF\alpha_\gamma(t,x,\xi) =&~ \Sph{1} \pi^{-\frac{1}{2}} e^{-\frac{x^2 + \xi^2}{2}}  \left(\cos\frac{\gamma}{2} + \sin\frac{\gamma}{2} e^{2\mi t} \frac{1}{\sqrt{2}}\left(\frac{\left(x+\xi\right)^2}{2}-1\right)\right), \\
\Sph{1}\alpha_\gamma(t,x,v) =&~ \Sph{1} \pi^{-\frac{1}{2}} e^{-\frac{x^2 + v^2}{2}} \left(\cos\frac{\gamma}{2} + \sin\frac{\gamma}{2}~ \frac{1}{2\sqrt 2} \left(x+\mi v\right)^2\right), \\
f_\gamma(t,x,v) =&~ \pi^{-1} e^{-\left(x^2+v^2\right)} \left(\cos^2\frac{\gamma}{2} + \sin^2\frac{\gamma}{2}~ \frac{1}{2} \left(x^2+v^2\right)^2 + \sin\gamma~ \frac{1}{\sqrt {2}}~ \Re e^{2\mi t} \left(x+\mi v\right)^2\right), \\
\rho_\gamma(t,x) =&~ \pi^{-\frac{1}{2}} e^{-x^2} \left(\cos^2\frac{\gamma}{2} + \sin^2\frac{\gamma}{2}~ \frac{1}{2} \left(x^4+x^2+\frac{3}{4}\right) + \sin\gamma~ \sqrt{2} \cos(2t) \left(x^2 - \frac{1}{2}\right)\right).
\end{align*}
\end{exam}

\section{Final comments}

Although the explicit results obtained for the very specific (and physically uninteresting) Harmonic Vlasov system are not very surprising and do not easily generalize to other systems, there are some important lessons to be learnt.

At first, this is the first time that existence of periodic solutions for any classical Vlasov system without boundary constraints could be found. Especially the Hamiltonian formalism of the Hamiltonian Vlasov system has proven its value in this area. 

Secondly, the method of symplectic symmetry reduction for phase equivariant Hamiltonian PDEs seems irreplaceable in the search for periodic solutions. The explicit solutions of Harmonic Vlasov prove that without this symmetry reduction, the trajectories are rarely closed and cannot be found by simply trying to solve the unreduced equation
\begin{equation*}
\lambda~ \partial_t\beta(t) - \HamVec{\HHF}(\beta) = 0, \quad \beta: \RN/\IN \to \Lp{2}{}, \quad \lambda > 0, 
\end{equation*}
around some equilibrium. Nevertheless, this seems to be a different story on the quotient manifold, where all periods predicted by the spectrum of the reduced vector field at the relative equilibrium could be naturally identified in periodic families bifurcating at that equilibrium.

Thirdly, the method is not able to deal with the translation invariance of the Hamiltonian Vlasov system in $\bx$ and $\xi$. They contribute to the kernel of the spectrum at critical points, a tricky point in applying general bifurcation theorems. While the systematic approach of Marsden-Weinstein seems to be valid also in this case, as indicated by the $4\ddim$ dimensional complement in Theorem \ref{thm:harmonic:spectral-parameterization}, the heavily used technical assumption of a smooth group action is not easily removed.

At last, the method encourages the discussion of other systems around known equilibria as many Hamiltonian PDEs show the structure of global $\Sph{1}$ phase invariance. Probably the largest obstacle is to find a representation of the equilibrium explicit enough to compute the spectrum of the first derivative of the reduced Hamiltonian vector field.

\section*{Conflict of interest statement}

The author declares that there are no conflicts of interest, because this work has not been funded by third parties.

\end{document}